\documentclass[12pt]{article}

\usepackage{t1enc}
\usepackage{hyperref}
\usepackage{amsmath,amstext,amscd,amsfonts,amsthm}
\usepackage{amssymb}
\usepackage{geometry}
\usepackage{enumerate}
\usepackage{tikz}
\usetikzlibrary{arrows}
\usepackage{centernot,cancel}
\usepackage{authblk}
\usepackage{color}
\allowdisplaybreaks

\newtheorem{teorema}{Theorem}

\newtheorem{definicion}{Definition}

\newtheorem{lema}{Lemma}

\theoremstyle{remark}
\newtheorem{ejemplo}{Example}
\newtheorem{observacion}{Observation}

\newcommand{\denselist}{\topsep 0pt\itemsep 0pt}

\newcommand{\set}[1]{\left\{ #1 \right\}}
\newcommand{\setdef}[2]{\set{ #1 \, : \, #2}}

\makeindex

\title{Perfectly nested circuits}

\author[1]{Mar\'ia Carrasco ($\dagger$)}
\author[2]{Zenaida Castillo}
\author[3]{Nerio Borges}
\affil[1]{Universidad Central de Venezuela}
\affil[2,3,4]{Yachay Tech}
\author[4]{Ram\'on Pino P\'erez}

\date{\today}

\begin{document}

\maketitle

\begin{abstract}
Nested graphs have been used in different applications, for example to represent knowledge in semantic networks. On the other hand, graphs with cycles are really important in surface reconstruction, periodic schedule  and network analysis. Also, of particular interest are the cycle basis, which arise in mathematical and algorithm problems. In this work we develop the concept of perfectly nested circuits, exploring some of their properties. The main result establishes an order isomorphism between some sets of perfectly nested circuits and equivalence classes over finite binary sequences.
\end{abstract}

\section{Introduction}

For decades several authors have used conceptual structures in research related with logic, linguistics and artificial intelligence. In \cite{Sowa76} Sowa presented the conceptual graph structure as a way to interpret questions and assertions in natural language to make inference in relational databases. The concept was developed in detail in \cite{Sowa84}.\\
Nested graphs were introduced and used for representing knowledge, see for example \cite{Leh92, Che_Mug92, Helbig06}. Given the notion of simple conceptual graphs the authors in \cite{Chein98} present the nested conceptual graphs to represent complex information. They show how nested graphs can be more convenient to represent  knowledge.\\
This article studies the characteristic of a similar structure that we have called Perfectly Nested Circuit (PNC). A mathematical framework is presented to analyze the properties of this kind of graphs. An interesting consequence of these properties is the existence of an order isomorphism between PNCs and equivalence classes of finite binary sequences. \\
This paper is organized as follows:
In section \ref{sec:Reducciones} we define reductions and the order induced by them, in a setting slightly more general than PNCs.

Section \ref{sec:CPA} defines perfectly nested circuits and studies the properties required to follow the results in Section \ref{sec:SucesionesBinarias&CPA}, where we state and prove the main result.

Finally in Section \ref{sec:Conclusiones} we summarize the obtained results and discuss some possible extensions to this work.


\section{Reductions on circuits}\label{sec:Reducciones}

 In this section we define two kind of operations on graphs\footnote{Actually, simple graphs are the only graphs considered in this work.}
 that will be central to our work.
We mostly follow the notation and terminology on graph theory from \cite{Diestel}.
We also use the following notation regarding binary relations \cite{book:Baader-Nipkow-rewriting}:
\begin{description}
	\item{$\stackrel{0}{\rightarrow}$} denotes the identity relation.
	\item{$\stackrel{i+1}{\rightarrow}$} is the $(i+1)$-fold composition of $\rightarrow$ with itself.
	\item{$\stackrel{+}{\rightarrow}$} is the {\em transitive closure} of $\rightarrow$ i.e.
	the smallest transitive relation containing $\rightarrow$.
	\item{$\stackrel{*}{\rightarrow}$} is the smallest reflexive and transitive relation containing $\rightarrow$,
	we call it the {\em transitive, reflexive, closure} of $\rightarrow$. Notice that
	$
	\stackrel{*}{\rightarrow}\,=\,\stackrel{+}{\rightarrow}\cup \stackrel{0}{\rightarrow}
	$.
\end{description}
The binary relation $\rightarrow$  is {\em terminating} if there is no infinite sequence $a_0,a_1,\dots$ such that $a_0\rightarrow a_1\rightarrow a_2\rightarrow \dots$

\begin{definicion}
	Let $\gamma=v_0v_1\dots v_n$ be a circuit.
	If there are $i\neq j$ such that $v_i=v_j$ we say that
	$\gamma'=v_iv_{i+1}\dots v_j$ is a {\em sub-circuit} of $\gamma$.
	If $(i,j)\neq (0,n)$ then $\gamma'$ is a {\em proper} sub-circuit.
	The sub-circuit $\gamma'$ is denoted by $[i,j]$.
\end{definicion}

{
	\begin{definicion}
		Let $\gamma=v_0v_1\dots v_n$ be a circuit with a proper sub-circuit $v_i\dots v_j$.
		The {\em internal reduction} of $\gamma$ at $i$ is the closed walk
		\begin{enumerate}
			\item $v_0\dots v_iv_{j+1}\dots v_n$ for $i\neq 0$,
			\item $v_j\dots v_n$ for $i=0$.
		\end{enumerate}	
		
		The {\em external reduction} of $\gamma$ at $i$ is the closed walk $v_iv_{i+1}\dots v_j$.
	\end{definicion}
}

\begin{observacion}\label{obs:CompactionsAndSetDifferences}
	Notice that if $[i,j]$ is a proper sub-circuit of $\gamma$:
	\begin{enumerate}
		\item The internal reduction of $\gamma$ at $i$ is $\gamma-\set{v_{i+1},\dots,v_{j-1}}$.
		\item The external reduction of $\gamma$ at $i$ is $\gamma-\set{v_0,\dots v_{i-1},v_{j+1},\dots, v_{n-1}}$
		if $i\neq 0$ and $\gamma-\set{v_{j+1},\dots, v_{n-1}}$ if $i=0$.
	\end{enumerate}
	In all the cases a reduction of $\gamma$ has less vertices than $\gamma$.
\end{observacion}

{
	\begin{lema}\label{le:ReduccionesPreservanCircuitos}
		If $\delta$ is a reduction of a circuit $\gamma$ then it is a circuit.
	\end{lema}
	
	\begin{proof}
		Any external reduction of $\gamma$ is a circuit since it is by definition a sub-circuit.
		
		Now suppose $\gamma=v_0v_1\dots v_n$ and $\delta$ is an internal reduction of $\gamma$ at some $i$.
		There are two possible cases. In each case we want to check two conditions: that $\delta$ is a closed walk and that every edge in $\delta$ is also an edge of $\gamma$.
		Then every repeated edge in $\delta$ is necessarily a repeated edge in $\gamma$, hence if $\gamma$ is a circuit
		$\delta$ is a circuit too.
		
		\begin{enumerate}
			\item If $i\neq 0$ then $\delta$ is closed since $\delta=v_0\dots v_iv_{j+1}\dots v_n$ and $v_0=v_n$ by hypothesis.
			
			On the other hand it is immediate that any edge $v_kv_{k+1}$ with $1\leq k<i$ or $j<k<n$ is an edge in $\gamma$. The edge $v_iv_{j+1}$ is also in $\gamma$ because $v_i=v_j$ thus $v_iv_{j+1}=v_jv_{j+1}$.
			
			\item If $i=0$ then $\delta=v_j\dots v_n$ and it is is closed
			since $v_0=v_j=v_n$. This also implies that the edges $v_nv_j$ and $v_nv_0$ are the same, hence $v_nv_j$ is an edge in $\gamma$. Any other edge in $\delta$ is clearly in $\gamma$.
			
		\end{enumerate}
		
		Therefore any reduction of a circuit, whether internal or external is again a circuit.
		
	\end{proof}
}

\begin{definicion}
	Given two circuits $\gamma,\delta$ the binary relation $\rightarrow_C$ is defined as:
\begin{equation}
	\gamma\rightarrow_C\delta\iff \delta\quad\text{is a reduction of $\gamma$ at some $i$}
\end{equation}
\end{definicion}	

Since $\gamma\rightarrow_C\delta$ implies $\delta=\gamma-U$
for some non empty subset $U$ of vertices of $\gamma$
it is clear that
\[
\gamma\rightarrow_C\delta\implies |\delta|<|\gamma|
\]
Hence the relation $\rightarrow_C$ terminates.

\begin{definicion}
	The {\em family of reductions} of a ciruit $\gamma$ is the set
	\[
	X_\gamma=\setdef{\delta\subseteq\gamma}{\gamma\stackrel{*}{\rightarrow}_C\delta}
	\]
\end{definicion}

\begin{teorema}
	If $\gamma$ is a circuit then each element of $X_\gamma$ is a circuit.
\end{teorema}

\begin{proof}
	Suppose $\delta\in X_\gamma$.
	Then by definition $\gamma\stackrel{*}{\rightarrow}_C\delta$
	i.e. there are a non negative integer $k$ and paths $\delta_0,\delta_1,\dots, \delta_k$
	such that
	\[\gamma=\delta_0\rightarrow_C\delta_1\rightarrow_C\dots\rightarrow_C\delta_k=\delta\]
	We prove by induction in $k$ that $\delta$ is a circuit.
	
	If $k=0$ then $\gamma=\delta$ thus $\delta$ is a circuit.
	Notice that $k=1$ implies $\gamma\rightarrow_C\delta$ and by Lemma \ref{le:ReduccionesPreservanCircuitos},
	$\delta$ is a circuit.
	
	Suppose now that
	\[\gamma=\delta_0\rightarrow_C\delta_1\rightarrow_C\dots\rightarrow_C\delta_k\rightarrow_C\delta_{k+1}=\delta\]
	By Inductive Hypothesis $\delta_k$ is a circuit and by Lemma \ref{le:ReduccionesPreservanCircuitos}
	we have that $\delta=\delta_{k+1}$ is a circuit.
	
	Therefore every element in the family of reductions of $\gamma$ is a circuit.
	
\end{proof}

\begin{lema}\label{le:XGammaEsTransitivo}
	If $\gamma$ is a circuit and $\delta\in X_\gamma$ then $X_\delta\subseteq X_\gamma$.
\end{lema}

\begin{proof}
	Immediate.
\end{proof}

\section{Perfectly Nested Circuits}\label{sec:CPA}
Nested circuits can be considered an extension on conceptual graphs \cite{Chein98}. Thus, perfectly nested circuits make reference to a kind of nested circuits with a particular structure that could be of interest in knowledge representation and manipulation of databases. \\
In this section we present and develop the concept of PNC in an appropriate context.

\begin{definicion}
	Let $\gamma=v_0v_1\dots v_n$ be a circuit.
	A pair $(i,j)$ with $0< i<j< n$ is an {\em intersection}
	of $\gamma$ if $v_i=v_j$.
	
	We denote by
	$I_\gamma$ the set of all the intersections of $\gamma$.
\end{definicion}

\begin{definicion}
	Define the function
	\begin{equation}
	\varphi: I_{\gamma}\longrightarrow V_G
	\end{equation}
	assigning the intersection $(i,j)$ to the vertex $v=v_i=v_j$.
	
	If $v=\varphi(i,j)$ then the vertex $v$ is {\em associated} to the intersection $(i,j)$.
	
	The vertex $v$ is {\em internal} if it is associated to some intersection $(i,j)$.
\end{definicion}

	Notice that set of all the internal vertices of $\gamma$ is $\varphi(I_\gamma)$.


\begin{figure}
	\begin{center}
		\begin{tikzpicture}[scale=.6]

		\coordinate (v0) at (-6,2);
		\coordinate (v1) at (-2,6);
		\coordinate (v2) at (2,6);
		\coordinate (v3) at (-2,4);
		\coordinate (v4) at (-4,0);
		\coordinate (v5) at (0,-2);
		\coordinate (v6) at (4,-2);
		\coordinate (v7) at (-2,0);
		\coordinate (v8) at (0,4);
		\coordinate (v9) at (2,4);
		\coordinate (v11) at (4,2);
		\coordinate (v13) at (6,2);
		\coordinate (v14) at (6,-2);
		\coordinate (v15) at (2,-4);
		\coordinate (v16) at (-2,-4);
		\coordinate (v17) at (-6,-2);
		
		\draw[fill=black] (v0) circle (2pt) node[left] {$v_0=v_{18}$};
		\draw[fill=black] (v1) circle (2pt) node[left] {$v_1$};
		\draw[fill=black] (v2) circle (2pt) node[above] {$v_2=v_{12}$};
		\draw[fill=black] (v3) circle (2pt) node[above] {$v_3$};
		\draw[fill=black] (v4) circle (2pt) node[left] {$v_4$};
		\draw[fill=black] (v5) circle (2pt) node[above] {$v_5$};
		\draw[fill=black] (v6) circle (2pt) node[below left] {$v_6=v_{10}$};
		\draw[fill=black] (v7) circle (2pt) node[above left] {$v_7$};
		\draw[fill=black] (v8) circle (2pt) node[above] {$v_8$};
		\draw[fill=black] (v9) circle (2pt) node[above] {$v_9$};
		\draw[fill=black] (v11) circle (2pt) node[right] {$v_{11}$};
		\draw[fill=black] (v13) circle (2pt) node[right] {$v_{13}$};
		\draw[fill=black] (v14) circle (2pt) node[right] {$v_{14}$};
		\draw[fill=black] (v15) circle (2pt) node[above] {$v_{15}$};
		\draw[fill=black] (v16) circle (2pt) node[above] {$v_{16}$};
		\draw[fill=black] (v17) circle (2pt) node[left] {$v_{17}$};
		
		\draw (v0) -- (v1) -- (v2) -- (v3) -- (v4) -- (v5)-- (v6) -- (v7)-- (v8)-- (v9)	-- (v6)-- (v11)-- (v2)-- (v13)-- (v14)-- (v15)-- (v16)-- (v17)-- (v0);

		\end{tikzpicture}
		\caption{Perfectly Nested Circuit} \label{fig:CircuitoPerfectamenteAnidado}
	\end{center}
\end{figure}
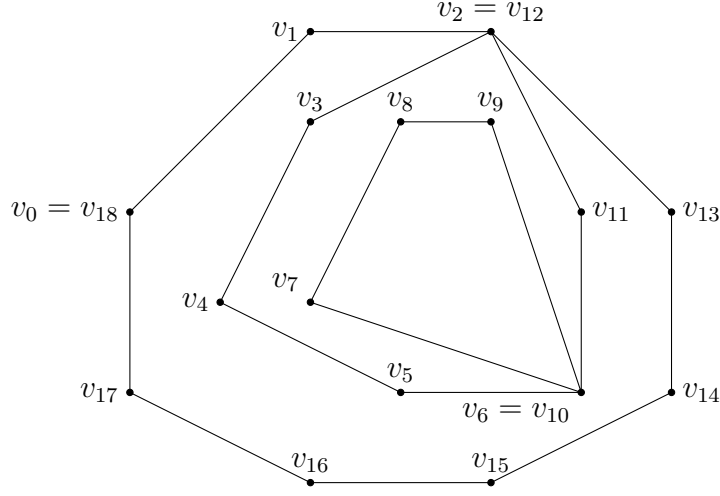

\begin{definicion}[Perfectly Nested Circuits]
	A circuit $\gamma=v_0v_1\ldots v_n$ is {\em perfectly nested} ({\em a PNC}) if
	it is either a simple cycle\footnote{A simple cycle is a cycle in which the vertices appear only once, except the beginning and the end of the cycle.} or there is a sequence
	\[
	0=k_0<k_1< k_2<\dots < k_m<k'_m<\dots <k'_1<k'_0=n
	\]
	such that:
	\begin{enumerate}
		\item $\varphi(I_\gamma)=\set{v_{k_1},\ldots,v_{k_m}}$.
		\item $\varphi(k_i,k_i')=v_{k_i}$.
		\item $i\neq j\implies v_{k_i}\neq v_{k_j}$.
	\end{enumerate}
	The sequence
	\[
	k_1,k_2,\dots,k_m,k'_m,\dots, k'_1
	\]
	is the {\em internal sequence} of $\gamma$.
\end{definicion}

 When $\gamma$ is not a simple cycle it is a {\em non trivial} PNC.

\begin{observacion}\label{obs:NestedSubcircuits}
Notice that sub-circuit $[k_{i+1},k_{i+1}']$ is a sub-circuit of the sub-circuit $[k_{i},k_{i}']$
for each $0\leq i<m$.
\end{observacion}

\begin{ejemplo}
	Figure \ref{fig:CircuitoPerfectamenteAnidado} shows a PNC $\gamma$.
	Given the sequence $k_1=2,k_2=6,k_2'=10,k_1'=12$:
	\begin{enumerate}
		\item The set of intersections is $I_\gamma=\set{(2,12),(6,10)}$ and their images are
		 $\varphi(2,12)=v_2$ and  $\varphi(6,10)=v_6$.
			Hence $\varphi(I_\gamma)=\set{v_{k_1},v_{k_2}}$.
		\item As we saw $\varphi(2,12)=v_2$ and $\varphi(6,10)=v_6$ 	
		i.e. $\varphi(v_{k_1},v_{k_1'})=v_{k_1}$ and $\varphi(v_{k_2},v_{k_2'})=v_{k_2}$.
		\item $v_{k_1}=v_2\neq v_{k_2}=v_6$.
	\end{enumerate}
	
\end{ejemplo}

\begin{definicion}
	Given a circuit $\gamma=v_0v_1\dots v_n$ the relation $<_I$
	is defined in $\set{0,\dots, n}^2$ by:
	\begin{equation}
	(i,j)<_I(k,\ell)\iff k< i<j<\ell
	\end{equation}
\end{definicion}

\begin{lema}
	The relation $<_I$ is a partial order on $\set{0,\dots, n}^2$
\end{lema}
\begin{proof}
	Straight forward.
\end{proof}

\begin{definicion}
	If $u=\varphi(i,j)$ and $v=\varphi(k,\ell)$ are two internal vertices in a PNC $\gamma$ then:
	\begin{equation}
	u\prec v\iff (i,j)<_I(k,\ell)
	\end{equation}
	if $u\prec v$ then $u$ is {\em more internal}
	than $v$ or, equivalently, we say $v$ is {\em more external}
	than $u$.
\end{definicion}

This relation $\prec$ is well defined because $\gamma$ is a PNC,
so if $v$ is an internal vertex then there is a unique $i$ such that
$v=\varphi(k_i,k'_i)$.

\begin{lema}\label{le:InternalOrderCharacterization}
	Let $u$, $v$ be two vertices in a PNC $\gamma$
	with internal sequence
	\[k_1,\dots, k_m,k'_m,\dots, k'_1\]
	Then
	$u\prec v$ iff there are indexes $k_n<k_p$ in the internal sequence
	such that $u=k_p,v=k_n$.
\end{lema}

\begin{proof}
	Since $u$ and $v$ are internal vertices of $\gamma$
	there are indexes $k_p,k_n$ in the internal sequence
	such that $u=\varphi(k_p,k_p')$ and $v=\varphi(k_n,k_n')$.
	If $u\prec v$ then by definition $(k_p,k_p')<_I(k_n,k_n')$
	thus $k_n<k_p$.
	
	On the other hand if $u=v_{k_p}$ and $v=v_{k_n}$ with $k_n<k_p$
	then $u=\varphi(k_p,k_p')$ and $v=\varphi(k_n,k_n')$.
	Since $k_n<k_p$ we have that $k_p'<k_n'$ hence $(k_p,k_p')<_I(k_n,k_n')$ and $u\prec v$.	
\end{proof}

\begin{definicion}
	If $\gamma$ is a PNC with internal sequence $k_1,k_2,\dots, k_m,k_m',\dots, k'_1$ then
	$v_{k_1}$ and $v_{k_m}$ are respectively its {\em outermost} and {\em innermost} vertices.
\end{definicion}

\begin{teorema}
	The binary relation `$\prec$' is a total order on the set of the internal vertices of a PNC $\gamma$.
\end{teorema}

\begin{proof}
	Let $k_1,\dots, k_m,k'_m,\dots, k'_1$
	be the internal sequence of $\gamma$.
	We check the order properties:
	\begin{enumerate}
		\item Irreflexivity: Suppose $v$ is an internal vertex. Then $v=\varphi(k_i,k_i')$ for some index $k_i$
		and $v\nprec v$ since $(k_i,k_i)\centernot<_I(k_i,k_i)$.
		
		\item Antisymmetry:Let $u,v$ be internal vertices of $\gamma$
		such that $u\prec v$.
		By Lemma \ref{le:InternalOrderCharacterization} there are indexes $k_i,k_j$ in the internal sequence
		with $k_j<k_i$ such that $u=v_{k_i}$ and $v=v_{k_j}$.
		Since $k_i\nless k_j$ we have, again by Lemma \ref{le:InternalOrderCharacterization},
		that $v\nprec u$.

		\item Transitivity: If $u,v,w$ are three internal vertices,
		then there are indices $k_i,k_j,k_\ell$ such that
		$u=v_{k_i}, v=v_{k_j}, w=v_{k_\ell}$.
		
		If $u\prec v$ and $v\prec w$ then, by Lemma \ref{le:InternalOrderCharacterization},
		follows that $k_j<k_i$ and $k_\ell<k_j$. Hence $k_\ell<k_i$ and $u\prec w$ by Lemma \ref{le:InternalOrderCharacterization}.

		\item Trichotomy:
		Suppose $u,v$ are two different internal vertices,
		thus $u=v_{k_i}$ and $v=v_{k_j}$ with $k_i\neq k_j$.
		Hence $k_i<k_j$ or $k_j<k_i$. If $k_i<k_j$ then $v\prec u$ and if $k_j<k_i$ then $u\prec v$.
		It is immediate that it is impossible $k_j<k_i$ and $k_i<k_j$ at the same time,
		hence one and only one of the following is satisfied
		\[
		u=v\quad \text{or}\quad u\prec v\quad \text{or}\quad v\prec u
		\]
		
	\end{enumerate}
\end{proof}


\begin{teorema}\label{teo:ReductionsPreservePncProperty}
	If $\gamma$ is a PNC then each element of $X_\gamma$ is a PNC.
\end{teorema}

\begin{proof}
	It is enough to prove that if
	$\gamma\rightarrow_C\delta$ and $\gamma$ are cnf then $\delta$ is a cnf.
	
	$\gamma$  is a PNC iff it has an internal sequence
	$k_1,\dots, k_m,k'_m,\dots, k'_1$ such that
	\begin{enumerate}
		\item A vertex $v$ is internal iff $v=v_{k_j}$ for some $k_j$,\label{it:InternalSequenceProperty1}
		\item $v_{k_j}=v_{k'_j}$,\label{it:InternalSequenceProperty2}
		\item the only proper sub-circuits of $\gamma$ are of the form $[k_j,k_j']$ and \label{it:InternalSequenceProperty3}
		\item $v_{k_j}\neq v_{k_\ell}$ if $j\neq \ell$.\label{it:InternalSequenceProperty4}
	\end{enumerate}
	
	Suppose that $\gamma\rightarrow_C\delta$. We have to examine two cases:
	
	If $\delta$ is an internal reduction of $\gamma$
	then it necessarily must be the reduction at $k_i$ for some $i$ with
	$1\leq i\leq m$ because the only proper sub-circuits of $\gamma$ have the form $[k_i,k_i']$.
	In this case we have, due to Observation \ref{obs:CompactionsAndSetDifferences}, that
	\[
	\delta=\gamma-\set{v_{k_i+1},\dots v_{k'_i-1}}
	\]
	Notice this operation preserves properties \ref{it:InternalSequenceProperty1} to \ref{it:InternalSequenceProperty4}
	thus $\delta$ is a PNC.
	
	If $\delta$ is an external reduction, then it is a proper sub-circuit of $\gamma$,
	hence it is $\delta=v_{k_i}v_{k_i+1}\dots v_{k_i'}$ for some $i$.
	As a consequence of $\gamma$ being a PNC:
	\begin{enumerate}
		\item The internal vertices of $\delta$ are $v_{k_{i+1}},v_{k_{i+2}},\dots v_{k_m}$,
		\item $v_{k_j}=v_{k'_j}$ for each $i<j\leq m$,
		\item the only proper sub-circuits of $\delta$ have the form $[k_j,k_j']$ with $i<j\leq m$, and
		\item $v_{k_j}\neq v_{k_\ell}$ whenever $j\neq \ell$.
	\end{enumerate}
	Therefore $\delta$ is a PNC.
	
	Hence if $\gamma$ is a PNC and $\gamma\rightarrow_C\delta$ then $\delta$ is a PNC.
	Thus if $\gamma$ is a PNC and $\gamma\stackrel{*}{\rightarrow}_C\delta$ it is easy to prove with an inductive argument that
	$\delta$ is a PNC.
	
	Therefore every element of $X_\gamma$ is a PNC.
\end{proof}

\begin{definicion}
	Given a circuit $\gamma$ in a graph, we define the binary relation $\leq_\gamma$
	on $X_\gamma$:
	\[
	\xi \leq_\gamma \eta \iff \eta\stackrel{*}{\rightarrow}_C\xi
	\]
\end{definicion}

\begin{teorema}
	The relation `$\leq_\gamma$' is a partial order on $X_\gamma$.
\end{teorema}

\begin{proof}
	Reflexivity and transitivity are direct from the same properties for
	$\stackrel{*}{\rightarrow}_C$.
	To prove reflexivity notice that $\xi\stackrel{*}{\rightarrow}_C\xi$
	because $\stackrel{*}{\rightarrow}_C$ is reflexive thus $\xi\leq_\gamma\xi$
	for each $\xi\in X_\gamma$.
	
	For transitivity suppose that $\xi\leq_\gamma\eta$ and $\eta\leq_\gamma\mu$.
	Then $\eta\stackrel{*}{\rightarrow}_C\xi$ and $\mu\stackrel{*}{\rightarrow}_C\eta$.
	Since $\stackrel{*}{\rightarrow}_C$ is transitive, we have $\mu\stackrel{*}{\rightarrow}_C\xi$
	hence $\xi\leq_\gamma\mu$.
	
	It remains to prove antisymmetry.
	Suppose that $\xi\leq_\gamma\eta$ and $\eta\leq_\gamma\xi$.
	It implies that $\eta\stackrel{*}{\rightarrow}_C\xi$ and $\xi\stackrel{*}{\rightarrow}_C\eta$.
	We know by Observation \ref{obs:CompactionsAndSetDifferences} that
	$|\eta|\leq|\xi|$ and $|\xi|\leq |\eta|$ so $|\eta|=|\xi|$.
	As $\xi\stackrel{*}{\rightarrow}_C\eta$
	then it must be $\xi\stackrel{0}{\rightarrow}_C\eta$.
	Thus $\eta=\xi$.
	
	Henceforth $\leq_\gamma$ is a partial order on $X_\gamma$.
\end{proof}

\begin{observacion}\label{obs:PredecesoresEnElOrdenDeCircuitos}
	If $\gamma$ is a PNC and $\delta\in X_\gamma$:
	\begin{enumerate}
		\item If $\delta$ is a cycle then it is a minimal element in $\leq_\gamma$ i.e.
		it has no predecessors.
		\item If $\delta$ is not a cycle,  then its immediate predecessors are
		the circuits obtained by reducing $\delta$ at its outermost and
		innermost vertices.
		\item $\gamma$ is the maximum element in the order $\leq_\gamma$.	
	\end{enumerate}
\end{observacion}

\begin{observacion}\label{obs:PNCWithSingleInternalVertex}
	If $\gamma$ is the circuit $v_0v_1\dots v_n$ with exactly one intersection $(i,j)$ then it is
	a PNC with $v_i$ as its only internal vertex.
	
	Notice that there are exactly two cycles in the circuit: the internal and the external
	reductions of $\gamma$ at $i$.
\end{observacion}

\begin{ejemplo}
	We show that if $\gamma$ is not perfectly nested
	then each node can have more than two immediate predecessors in the order $\leq_\gamma$ .
	
	\begin{figure}
		\begin{center}
			\begin{tikzpicture}[scale=.6]
				
				
				\draw[fill=black] (-6,0) circle  (2pt) node[left] {$v_0=v_{18}$};
				\draw[fill=black] (-3,3) circle  (2pt) node[above left] {$v_1=v_5$};
				\draw[fill=black] (3,3) circle  (2pt) node[above] {$v_6$};
				\draw[fill=black] (6,0) circle  (2pt) node[right] {$v_7=v_{11}$};
				\draw[fill=black] (-3,-3) circle  (2pt) node[below left] {$v_{13}=v_{17}$};
				\draw[fill=black] (3,-3) circle  (2pt) node[below] {$v_{12}$};
				
				\draw[fill=black] (-1,2) circle (2pt) node[right] {$v_2$};
				\draw[fill=black] (-1.5,.5) circle (2pt) node[right] {$v_3$};
				\draw[fill=black] (-3,1) circle (2pt) node[below left] {$v_4$};

				\draw[fill=black] (3,1) circle (2pt) node[above] {$v_8$};
				\draw[fill=black] (2,.5) circle (2pt) node[left] {$v_9$};
				\draw[fill=black] (3,-1) circle (2pt) node[below] {$v_{10}$};
				
				\draw[fill=black] (-3,-1) circle (2pt) node[above left] {$v_{14}$};
				\draw[fill=black] (-1,-0.5) circle (2pt) node[right] {$v_{15}$};
				\draw[fill=black] (0,-2) circle (2pt) node[below right] {$v_{16}$};
				
				\draw (-6,0) -- (-3,3) -- (3,3) -- (6,0) -- (3,-3) -- (-3,-3) -- (-6,0);
				
				\draw (-3,3) -- (-1,2) -- (-1.5,.5) -- (-3,1) -- (-3,3);
				
				\draw (6,0) -- (3,1) -- (2,.5) -- (3,-1) -- (6,0);
				
				\draw (-3,-3) -- (-3,-1) -- (-1,-0.5) -- (0,-2) -- (-3,-3);
			\end{tikzpicture}
			\caption{Circuit $\gamma$} \label{fig:CircuitoGamma}
		\end{center}
	\end{figure}
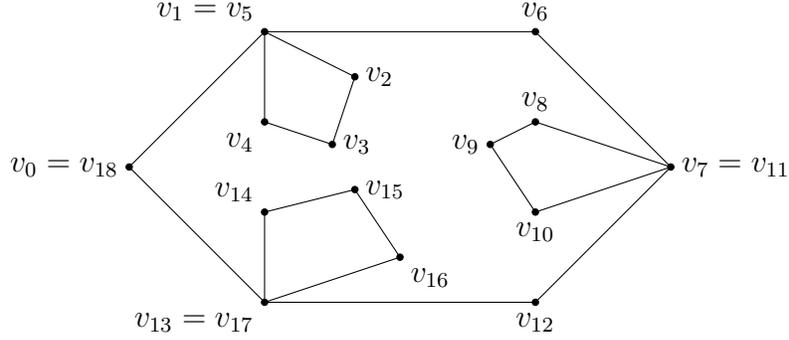

	Consider the circuit $\gamma=v_0v_1\dots v_{18}$ with $v_1=v_5$, $v_7=v_{11}$, $v_{13}=v_{17}$
and $v_0=v_{18}$ (see Figure \ref{fig:CircuitoGamma}).
The internal vertices of $\gamma$ are $v_1,v_7$ and $v_{13}$.
The sub-circuits of $\gamma$ are $v_1v_2v_3v_4v_5$, $v_7v_8v_9v_{10}v_{11}$ and $v_{13}v_{14}v_{15}v_{16}v_{17}$.
These three sub-circuits are disjoint thus $\gamma$ is not a PNC by Observation \ref{obs:NestedSubcircuits}.

\begin{figure}
	\begin{center}
		\begin{tikzpicture}
		\coordinate (g) at (0,4);
		
		\node at (0,4) {$\gamma$};
		\node at (-2,2) {$\gamma_1$};
		\node at (0,2) {$\gamma_2$};
		\node at (2,2) {$\gamma_3$};	
		\node at (-2,0) {$\gamma_4$};
		\node at (0,0) {$\gamma_5$};
		\node at (2,0) {$\gamma_6$};
		\node at (0,-2) {$\gamma_7$};
		\node at (-2,-2) {$\gamma_8$};
		\node at (1,-2)  {$\gamma_9$};
		\node at (2,-2)  {$\gamma_{10}$};
		
		\draw[- angle 90] (-.3,3.7) -- (-2,2.3);
		\draw[- angle 90] (0,3.7) -- (0,2.3);
		\draw[- angle 90] (.3,3.7) -- (2,2.3);
		
		\draw[- angle 90] (-2,1.7) -- (-.3,.3);
		\draw[- angle 90] (-1.7,1.7) -- (1.7,.3);
		
		\draw[- angle 90] (-.3,1.7) -- (-2,.3);
		\draw[- angle 90] (.3,1.7) -- (2,.3);
		
		\draw[- angle 90] (1.7,1.7) -- (-1.7,.3);
		\draw[- angle 90] (2,1.7) -- (0.3,.3);
		
		\draw[- angle 90] (-1.7,-.3) -- (-.3,-1.7);
		\draw[- angle 90] (0,-.3) -- (0,-1.7);
		\draw[- angle 90] (1.7,-.3) -- (.3,-1.7);
		
		\draw[- angle 90] (-2,-.3) -- (-2,-1.7);
		\draw[- angle 90] (.3,-.3) -- (1,-1.7);
		\draw[- angle 90] (2,-.3) -- (2,-1.7);
		\end{tikzpicture}
		\caption{Diagram of $\rightarrow_C$}\label{fig:DiagramaReducciones}
	\end{center}
\end{figure}
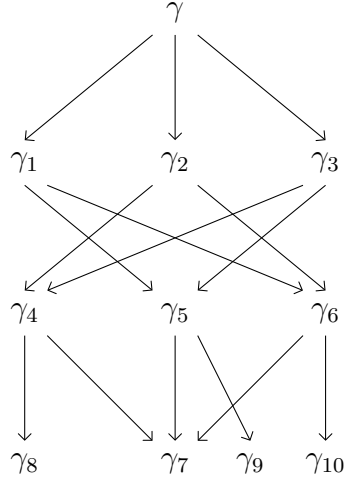

Label the circuits $v_1v_2v_3v_4v_5$, $v_7v_8v_9v_{10}v_{11}$ and $v_{13}v_{14}v_{15}v_{16}v_{17}$
as $\gamma_8$, $\gamma_9$ and $\gamma_{10}$ respectively.
These three sub-circuits are the only external reductions of $\gamma$. The internal reductions of $\gamma$ are:
\begin{align*}
\gamma		&	=v_0\underbrace{v_1v_2v_3v_4v_5}_{\gamma_8}v_6\underbrace{v_7v_8v_9v_{10}v_{11}}_{\gamma_9}
v_{12}\underbrace{v_{13}v_{14}v_{15}v_{16}v_{17}}_{\gamma_{10}}v_{18}\\
\gamma_1	&	=v_0v_1	v_6\underbrace{v_7v_8v_9v_{10}v_{11}}_{\gamma_9}
v_{12}\underbrace{v_{13}v_{14}v_{15}v_{16}v_{17}}_{\gamma_{10}}v_{18}\\
\gamma_2		&	=v_0\underbrace{v_1v_2v_3v_4v_5}_{\gamma_8}v_6v_7
v_{12}\underbrace{v_{13}v_{14}v_{15}v_{16}v_{17}}_{\gamma_{10}}v_{18}\\	
\gamma_3		&	=v_0\underbrace{v_1v_2v_3v_4v_5}_{\gamma_8}v_6\underbrace{v_7v_8v_9v_{10}v_{11}}_{\gamma_9}
v_{12}v_{13}v_{18}\\	
\gamma_4		&	=v_0\underbrace{v_1v_2v_3v_4v_5}_{\gamma_8}v_6v_7
v_{12}v_{13}v_{18}\\
\gamma_5		&	=v_0v_1v_6\underbrace{v_7v_8v_9v_{10}v_{11}}_{\gamma_9}
v_{12}v_{13}v_{18}\\
\gamma_6	&	=v_0v_1	v_6v_7
v_{12}\underbrace{v_{13}v_{14}v_{15}v_{16}v_{17}}_{\gamma_{10}}v_{18}\\	
\gamma_7	&	=v_0v_1	v_6v_7
v_{12}v_{13}v_{18}
\end{align*}

Figure \ref{fig:DiagramaReducciones} shows a diagram of the relation $\rightarrow_C$.
Since $\leq_\gamma$ is the inverse relation of $\stackrel{*}{\rightarrow}_C$
it is clear that $\gamma$ has three immediate predecessors.
\end{ejemplo}

\subsection{A characterization of PNC}

In this section we characterize PNC as chains of cycles.

{
	
	\begin{definicion}
		Suppose $C_0,C_1,\dots, C_m$ are cycles:
		\[
		C_j:= v^j_0v^j_1\dots v^j_{n_j}
		\]
		with two different distinguished vertices $u_j$ and $v_j$
		for each $0\leq j\leq m$ such that $u_0\neq v^0_0$.
		
		We denote as
		\[
		C_0C_1\dots C_m
		\]
		the graph built by the identification of $u_j$ with $v_{j+1}$ for $0\leq j<m$
		
		Such a graph is a {\em chain of cycles} (or just a {\em chain} when there is no risk of confusion).
		We call each cycle $C_j$ with $0\leq j\leq m$ a {\em link}.
	\end{definicion}

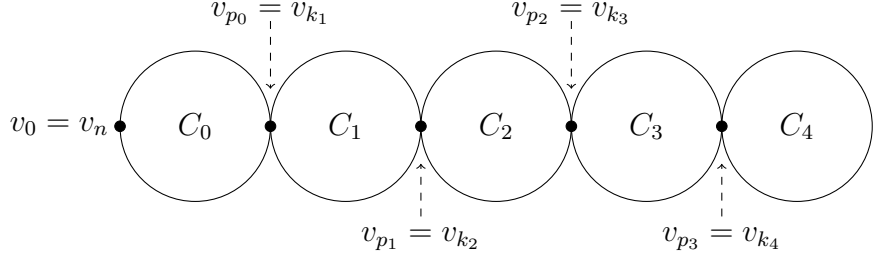
\begin{figure}
	\begin{center}
		\begin{tikzpicture}
		\draw (0,0) circle (1cm);
		\draw (2,0) circle (1cm);
		\draw (4,0) circle (1cm);
		\draw (6,0) circle (1cm);
		\draw (8,0) circle (1cm);
		
		
		\draw[fill=black] (-1,0) circle	(.07cm) node[left] {$v_0=v_n$};
		\draw[fill=black] (1,0) circle	(.07cm);
		\draw[fill=black] (3,0) circle	(.07cm);
		\draw[fill=black] (5,0) circle	(.07cm);
		\draw[fill=black] (7,0) circle	(.07cm);
		
		
		\node at (0,0) {$C_0$};
		\node at (2,0) {$C_1$};
		\node at (4,0) {$C_2$};
		\node at (6,0) {$C_3$};
		\node at (8,0) {$C_4$};

		
		\node at (1,1.5) {$v_{p_0}=v_{k_1}$};
		\draw[dashed, ->] (1,1.4) -- (1,.5);
		
		\node at (3,-1.5) {$v_{p_1}=v_{k_2}$};
		\draw[dashed, ->] (3,-1.2) -- (3,-.5);
		
		\node at (5,1.5) {$v_{p_2}=v_{k_3}$};
		\draw[dashed, ->] (5,1.4) -- (5,.5);
		
		\node at (7,-1.5) {$v_{p_3}=v_{k_4}$};
		\draw[dashed, ->] (7,-1.2) -- (7,-.5);

		\end{tikzpicture}
		\caption{A chain of 5 cycles} \label{fig:CadenaDeCiclos}
	\end{center}
\end{figure}	
	
	\begin{lema}\label{le:ChainsOfCyclesArePNCs}
		If $C_0,C_1,\dots, C_m$ are cycles
		\[
		C_j:= v^j_0v^j_1\dots v^j_{n_j}
		\]
		with two different distinguished vertices $u_j$ and $v_j$
		for each $0\leq j\leq m$ such that $u_0\neq v^0_0$
		then the chain $C_0C_1\dots C_m$ is a PNC
		whose internal vertices are $v^1_{k_1},v^2_{k_2},\ldots, v^m_{k_m}$.
	\end{lema}

\begin{proof}
	The proof proceeds by induction in the number of cycles.
	
	Our base case is $m=1$. We have the circuits
	\begin{align*}
	C_0&:=v^0_0v^0_1\dots v^0_{n_0}\\
	C_1&:=v^1_0v^1_1\dots v^1_{n_1}
	\end{align*}
	With $u_0=v^0_p$ and $v_1=v^1_k$ for some $0< p<n_0$ and some $0\leq k<n_1$.
	If we identify $v^0_p$ with $v^1_k$ the resulting graph $C_0C_1$ consists of:
	\begin{enumerate}\denselist
		\item A path within $C_0$ joining $v_0^0$ with $v^0_{p_0-1}$,
		\item an edge joining $v^0_{p_0-1}$ and $v^1_{k_1}$ (which now replaces $v^0_{p_0}$),
		\item the cycle $C_1$ joining $v^1_{k_1}$ with itself,
		\item an edge joining $v^1_{k_1}$ and $v^0_{p_0+1}$, and
		\item a path within $C_0$ joining $v^0_{p_0+1}$ and $v^0_0$.
		
	\end{enumerate}
	This graph is a closed walk. Moreover, since $C_0$ and $C_1$ are cycles,
	there are no repeated edges and
	the only repeated vertex is $v^1_{k_1}$ thus $C_0C_1$ is a PNC
	by Observation \ref{obs:PNCWithSingleInternalVertex}.

	Now suppose	we have cycles $C_0,C_1,\ldots, C_m$
	\[
	C_j:= v^j_0v^j_1\dots v^j_{n_j}
	\]
	with two different distinguished vertices $u_j$ and $v_j$
	for each $0\leq j\leq m$ such that $u_0\neq v^0_0$.
	In each cycle $C_j$ we have that $u_j=v^j_{p_j}$ and $v_j=v^j_{k_j}$ for some
	indexes $p_j\neq k_j$ from $\set{0,1,\dots, n_j}$.
	By inductive hypothesis, the graph $C_0C_1\dots C_{m-1}$
	is a PNC with internal vertices $v^1_{k_1},v^2_{k_2},\dots, v^{m-1}_{k_{m-1}}$.
	
	If we join the cycle $C_m$ with $C_0C_1\dots C_{m-1}$ by identification of $v^{m-1}_{p_{m-1}}$ with
	$v^m_{k_m}$ then
	we obtain a graph \(C_0C_1\dots C_{m-1}C_m\). Since we are adding a closed walk, this new graph
	is a closed walk too. Moreover, this graph is a circuit because $C_m$ is a cycle and
	also by the very same reason $v^m_{k_m}$ is the only added repeated vertex.
	Hence \(C_0C_1\dots C_{m-1}C_m\) is a PNC with internal vertices
	$v^1_{k_1},v^2_{k_2},\ldots, v^m_{k_m}$.	
\end{proof}

{
	\begin{lema}\label{le:PNCsAreChainsOfCycles}
		If $\gamma$ is a PNC with internal vertices $v_{k_1},v_{k_2},\dots v_{k_m}$ then
		there are cycles $C_0,C_1,\dots, C_m$ each of them with a pair of distinguished different vertices
		$u_j,v_j$, such that $\gamma=C_0C_1\dots C_m$. These cycles are unique up to isomorphisms.
	\end{lema}
}
\begin{proof}
	We prove this by induction on $m$.
	
	The base case is $m=1$. Consider a PNC $\gamma$ with only one internal vertex $v_{k_1}$
	and internal sequence $k_1,k'_1$.
	At $k_1$ this circuit has internal reduction $C_0=v_0v_1\dots v_{k_1}v_{k'_1+1}\dots v_n$
	and external reduction $C_1=v_{k_1}v_{k_1+1}\dots v_{k'_1}$. Both are cycles since $\gamma$
	has no other internal vertex and $\gamma=C_0C_1$ because they only have vertex $v_{k_1}$ in common.
	The uniqueness in this case is immediate since two cycles are isomorphic iff they have
	the same cardinality.
	
	Now suppose $\gamma$ is a PNC with internal vertices $v_{k_1},v_{k_2},\dots,v_{k_m}$.
	By Theorem \ref{teo:ReductionsPreservePncProperty} the internal reduction $\delta$ of $\gamma$ at $v_{k_m}$ is a PNC
	with $m-1$ internal vertices $v_{k_1},v_{k_2},\dots,v_{k_{m-1}}$ hence by inductive hypothesis
	there are cycles $C_0,C_1,\dots, C_{m-1}$, unique up to isomorphism, such that
	$\delta=C_0C_1\dots C_{m-1}$.
	On the other hand the external reduction of $\gamma$ at $v_{k_m}$
	is the cycle $v_{k_m}v_{k_m}+1\dots v_{k_m'}$ if we denote this cycle by $C_m$
	it is clear that $\gamma=C_0C_1\dots C_{m-1}C_m$.
	
	These cycles are unique because for any other cycle $C'_{m}$ with $C_m\ncong C_m'$,
	we have that $|\gamma|\neq |C_0C_1\dots C_{m-1}C'_m|$.

\end{proof}

Lemmas \ref{le:ChainsOfCyclesArePNCs} and \ref{le:PNCsAreChainsOfCycles} prove
the following:
{
\begin{teorema}\label{teo:APncIsAChain}
A circuit $\gamma$ is perfectly nested if and only if it is a chain
of cycles.
\end{teorema}

Theorem \ref{teo:APncIsAChain} shows that each PNC is a chain of cycles like the one presented
on Figure \ref{fig:CadenaDeCiclos}.

\begin{observacion}\label{obs:ReductionsOnChains}
	If $\gamma$ is a PNC $C_0C_1\dots C_m$ with internal vertices
	\[v_{k_1},v_{k_2},\dots,v_{k_m}\]
	where $v_{k_j}$ joins $C_{j-1}$ with $C_j$,
	then the internal reduction at $k_j$ is $C_0C_1\dots C_{j-1}$
	and the external reduction at $k_j$ is $C_jC_{j+1}\dots C_m$.
\end{observacion}


\begin{lema}\label{le:CaracterizacionFamiliaDeReducciones}
	Let $\gamma=C_0C_1\dots C_m$ be a PNC.
	A circuit $\delta$ is an element of $X_\gamma$ iff there are $j,\ell$ with $0\leq j<j+\ell\leq m$ such that $\delta=C_jC_{j+1}\dots C_{j+\ell}$ .
\end{lema}

\begin{proof}
	If $\delta\in X_\gamma$ then $\gamma\stackrel{*}{\rightarrow}_C\delta$ i.e. there are circuits $\gamma_0, \gamma_1,\dots, \gamma_n$ with $n\geq 0$ such that
	\[\gamma=\gamma_0\rightarrow_C\gamma_1\rightarrow_C\dots\rightarrow_C\gamma_n=\delta\]
	We will prove by induction on $n$ that $\delta=C_jC_{j+1}\dots C_{j+\ell}$ for some $j,\ell$.
	
	The base case $n=0$ is trivial since $\delta=\gamma$.
	
	Now Suppose \[\gamma=\gamma_0\rightarrow_C\gamma_1\rightarrow_C\dots\rightarrow_C\gamma_n=\delta\]
	for some $n>0$. Then by inductive hypothesis
	\[\gamma_{n-1}=C_{j}C_{j+1}\dots C_{j+\ell}\]
	for some pair $j,\ell$ with $0\leq j<j+\ell\leq m$. Since $\gamma_{n-1}\rightarrow_C\gamma_n=\delta$
	there is a $p$ with $j<k<j+\ell$ such that either
	\[\delta=C_{p}C_{k+1}\dots C_{j+\ell}\]
	if $\delta$ is the external reduction of $\gamma_{n-1}$ at ${k_p}$
	or
	\[\delta_{n-1}=C_{j}C_{j+1}\dots C_{p-1}\]
	if $\delta$ is the internal reduction of $\gamma_{n-1}$ at ${k_p}$ (see Observation \ref{obs:ReductionsOnChains}).
	
	On the other hand suppose that $\delta=C_{j}C_{j+1}\dots C_{j+\ell}$ and $\gamma\neq\delta$. We have three possible cases:
	\begin{enumerate}
		\item If $0<j<j+\ell<m$ then $\gamma\rightarrow_C\gamma_1\rightarrow_C\delta$ where $\gamma_1$ is the internal reduction of
		$\gamma$ at $k_{j+\ell}$ and $\delta$ is the external reduction of $\gamma_1$ at $k_j$.
		\item If $0=j$ then $\delta$ is the internal reduction of $\gamma$ at $k_\ell$.
		\item If $j+\ell=m$ the $\delta$ is the external reduction of $\gamma$ at $k_\ell$.
	\end{enumerate}
On each case $\gamma\stackrel{*}{\rightarrow}_C\delta$ thus $\delta\in X_\gamma$.
\end{proof}

}


\section{Binary sequences and perfectly nested circuits}\label{sec:SucesionesBinarias&CPA}

 In this section we describe, given some $m\in \mathbb N$, an equivalence relation on the of binary sequences
 less or equal than $m$. Then we give an order to the set of equivalence classes and finally
 we prove our main result namely
 that this ordered set is isomorphic to the family of reductions of a PNC with $m$ internal vertices.

\subsection{Binary sequences}

The length of a sequence $s$ is denoted as $\ell(s)$.
We denote the empty sequence by $\emptyset$.
The set of all the binary sequences of length at most $m$ is denoted by $2^{\leq m}$.
We denote as $s^\frown t$ the {\em concatenation} of sequences $s$ and $t$.
\begin{definicion}
	The pair $(s,t)$ of sequences in $2^{\leq m}$ belong to the relation $\sim_m$ iff:
	\begin{enumerate}
		\item $\ell(s)=\ell(t)$, and
		\item $s$ and $t$ have the same number of 1's (and consequently, the same number of 0's.
	\end{enumerate}
\end{definicion}

We will denote by $|s|_0$ and $|s|_1$ respectively the number of 0's and 1's in $s$.

\begin{observacion}
	Notice that $\sim_m$ is an equivalence relation.
	We denote the quotient set $2^{\leq m}/\sim_m$ as $S_m$.
\end{observacion}

\begin{definicion}
	We define a binary relation $\leq_m$ on $S_m$ as follows:
	$[s]\leq_m[t]$ iff there are sequences $s'\in[s],\, t'\in[t]$
	such that $s'$ extends $t'$.
\end{definicion}

\begin{teorema}
	The relation $\leq_m$ is an order on $S_m$.
\end{teorema}

\begin{proof}
	Let $[s]$ be an element of $S_m$. Since it extends itself it is clear that $[s]\leq_m[s]$, so
	$\leq_m$ is reflexive.
	
	To check antisymmetry, suppose that $[s],[t]$ are two equivalence classes in $S_m$
	such that $[s]\leq_m[t]$ and $[t]\leq_m[s]$.
	
	Since $[s]\leq_m[t]$
	there are sequences $s'\in[s]$ and $t'\in[t]$
	such that $s'$ extends $t'$ thus
	\[
	\ell(t)=\ell(t')\leq \ell (s')=\ell(s)
	\]
	and
	\[
	|t|_1=|t'|_1\leq |s'|_1=|s|_1
	\]
	And since $[t]\leq_m[s]$
	there are sequences $\tilde s\in[s]$ and $\tilde t\in[t]$
	$\tilde t$ extends $\tilde s$ so
	\[
	\ell(s)=\ell(\tilde s)\leq \ell(\tilde t)=\ell(t)
	\]
	and
	\[
	|s|_1=|\tilde s|_1\leq |\tilde t|_1=|t|_1
	\]
	Hence
	$\ell(s)=\ell(t)$ and $|s|_1=|t|_1$. As a consequence $s\sim_m t$ and therefore $[s]=[t]$.
	
	To prove transitivity suppose that $[r],[s],[t]$ are three classes in $S_m$
	such that $[r]\leq_m[s]$ and $[s]\leq_m[t]$.
	As $[r]\leq_m[s]$ there are sequences $r'\in[r]$ and $s'\in [s]$ with $r'$ extending $s'$
	thus $|s|_0=|s'|_0\leq |r'|_0=|r|_0$ and $|s|_1=|s'|_1\leq |r'|_1=|r|_1$.
	Analogously, since $[s]\leq_m[t]$, there are sequences $\tilde s\in[s]$ and $\tilde t\in[t]$ such that
	$\tilde s$ extends $\tilde t$.
	Hence $|t|_0=|\tilde t|_0\leq |\tilde s|_0=|s|_0$ and $|t|_1=|\tilde t|_1\leq |\tilde s|_1=|s|_1$,
	thus
	\[
	|t|_0\leq |r|_0\quad\text{y}\quad|t|_1\leq |r|_1
	\]
	and since $\ell(t)=|t|_0+|t|_1$ and $\ell(r)=|r|_0+|r|_1$ we have that
	\[
	t'=\underbrace{0,\ldots, 0}_{|t|_0},\underbrace{1,\dots, 1}_{|t|_1}
	\]
	is an element of $[t]$
	and
	\[
	\tilde r=\underbrace{0,\ldots, 0}_{|t|_0},\underbrace{1,\dots, 1}_{|t|_1},
	\underbrace{0,\ldots, 0}_{|r|_0-|t|_0},\underbrace{1,\dots, 1}_{|r|_1-|t|_1}
	\]
	is an element of $[r]$ that extends $t'$.
	Therefore $[r]\leq_m[t]$.
\end{proof}

\subsection{Order isomorphism between perfectly nested circuits and classes of binary sequences}

\begin{definicion}
	Let {$\gamma=C_0C_1\dots C_m$} be a PNC with internal sequence $k_1,k_2,\dots,k_m,k'_m,\dots, k'_1$.
	The {\em $0$-reduction} of $\gamma$ is the internal reduction of $\gamma$ at vertex $v_{k_m}$ i.e.
	the circuit
	\[
	{C_0C_1\dots C_{m-1}}
	\]
	and the {\em $1$-reduction} is the external reduction of $\gamma$ at vertex $v_{k_1}$ i.e.
	\[
	{C_1\dots C_m}
	\]
\end{definicion}

\begin{definicion}
	Suppose $\delta\in X_\gamma$. A {\em 0-1-sequence for $\delta$} is a sequence
	$\gamma_0,\gamma_1,\ldots, \gamma_p$ of elements in $X_\gamma$
	such that $\gamma_0=\gamma,\gamma_p=\delta$ and for each $0\leq i<p$,
	$\gamma_{i+1}$ is obtained from $\gamma_i$ either by a
	$0$-reduction or a $1$-reduction.
	
	If
	\begin{equation}\label{eq:Sucesion01}
	\gamma_0,\gamma_1,\ldots, \gamma_p	
	\end{equation}
	is a 0-1-sequence for $\delta$ where each $\gamma_{i+1}$ is the $0$-reduction of $\gamma_i$, then we say it is a
	{\em $0$-sequence}. If instead each $\gamma_{i+1}$ is the $1$-reduction of $\gamma_i$
	we say it is a {\em $1$-sequence}.
	
\end{definicion}

Observe that each 0-sequence and each 1-sequence are 0-1-sequences
as well
and that a concatenation of 0-1-sequences is again a 0-1-sequence.

\begin{lema}\label{le:ObtencionDeCircuitos}
	If $\gamma$ is a non trivial PNC and $\delta\in X_\gamma$ then there is a 0-1-sequence for $\delta$
	or, equivalently every element of $X_\gamma$ can be obtained from $\gamma$ by a sequence of
	0-reductions and 1-reductions.
\end{lema}

\begin{proof}
	Let $\gamma$ be a PNC with internal sequence $k_1,k_2,\dots,k_m,k'_m,\dots, k'_1$.
	We now by definition that for every PNC $\delta$:
	\[
	\delta\in X_\gamma\iff \gamma\stackrel{*}{\rightarrow}_C\delta
	\]
	so it suffices to check that if $\gamma\rightarrow_C\delta$ then
	there is a 0-1-sequence for $\delta$.
	
	If $\gamma\rightarrow_C\delta$ then $\delta$ is either the internal
	or the external reduction of $\gamma$ at some internal vertex $v_{k_i}$.
	
	Suppose that $\delta$ is the internal reduction of $\gamma$ at $v_{k_i}$.
	In this case:
	\[
	\delta=C_0C_1\dots C_{i-1}
	\]
	We take $j=m-i$ thus $k_i=k_{m-j}$.
	We prove by induction on $0\leq j<m$
	that there is a 0-sequence $\gamma_0,\gamma_1,\dots \gamma_{j+1}$ for $\delta$.
	
	If $j=0$ then $\delta$ is the internal reduction of $\gamma$ at $v_{k_m}$
	then by definition it is the $0$-reduction of $\gamma$.
	Then we have the $0$-sequence $\gamma_0,\gamma_1$ with $\gamma_0=\gamma$ and $\gamma_1=\delta$.
	
	Now suppose $\delta$ is the internal reduction of $\gamma$ at $v_{k_{m-(\ell+1)}}$ with
	$\ell+1<m$ i.e.
	\[
	\delta=C_0C_1\dots C_{(m-\ell)-2}
	\]
	Notice that $\delta$ is the $0$-reduction of
	\[
	\delta'=C_0C_1\dots C_{(m-\ell)-1}
	\]
	which in its turn is the internal reduction of $\gamma$ at $v_{k_{m-\ell}}$.
	By inductive hypothesis there is a 0-sequence $\gamma_0,\dots,\gamma_{\ell+1}$
	for $\delta'$.
	This implies the existence of a 0-sequence $\gamma_0,\dots,\gamma_{\ell+2}$ for $\delta$.
	
	As every 0-sequence is a 0-1-sequence we have a 0-1-sequence for $\delta$.
	
	On the other hand if $\delta$ is the external reduction of $\gamma$ at $v_{k_i}$:
	\[
	\delta=C_iC_{i+1}\dots C_m
	\]
	We prove the existence of a $1$-sequence $\gamma_0,\dots,\gamma_i$
	for $\delta$ by induction on $1\leq i\leq m$.
	
	Our base case is $i=1$. Then
	\[
	\delta=C_1C_{2}\dots C_m
	\]
	is the $1$-reduction of $\gamma$.
	The 1-sequence for $\delta$ is $\gamma_0,\gamma_1$ with $\gamma_0=\gamma$ and $\gamma_1=\delta$.
	
	Now suppose $i=\ell+1$ with $\ell<m$
	Then the external reduction of $\gamma$ at $v_{k_{\ell+1}}$ is
	\[
	\delta=C_{\ell+1}C_{\ell+2}\dots C_m
	\]
	which is the $1$-reduction of
	\[
	\delta'=C_\ell C_{\ell+1}C_{\ell+2}\dots C_m
	\]
	By inductive hypothesis there is a 1-sequence
	$\gamma_0,\gamma_1,\dots,\gamma_\ell$ for $\delta'$
	hence there is a $1$-sequence $\gamma_0,\gamma_1,\dots,\gamma_{\ell+1}$ for $\delta$.
	Since every 1-sequence is a 0-1-sequence, there is a 0-1-sequence for $\delta$.
	
	We just proved that $\gamma\rightarrow_C\delta$ implies there is a
	0-1-sequence $\gamma_0,\dots,\gamma_p$ for $\delta$.
	Since the concatenation of 0-1-sequences yields a 0-1-sequence,
	we conclude there is a 0-1-sequence for every
	$\delta$ such that $\gamma\stackrel{*}{\rightarrow}_C\delta$
	i.e. there is a 0-1-sequence for each $\delta\in X_\gamma$.
	By definition this means that each $\delta\in X_\gamma$ is
	obtained from $\gamma$ by a sequence of 0-reductions and 1-reductions.

\end{proof}

\begin{lema}\label{teo:TeoremaIndependencia}
	Let $\gamma=v_0v_1\dots v_n$ be a PNC with $m\geq 1$ internal vertices and $\delta\in X_{\gamma}$, $\delta\neq \gamma$.
	Let $\gamma,\gamma_1,\ldots, \gamma_p$ be a 0-1-sequence for $\delta$.
	
	If $p=p_0+p_1$ with $p_0$ the number of $0$-reductions
	and $p_1$ the number of $1$-reductions in the 0-1-sequence, then
	\begin{equation}\label{eq:TeoremaIndependencia}
	{\delta=\gamma_p=C_{p_1}C_{p_1+1}\dots C_{m-p_0}}
	\end{equation}
\end{lema}

\begin{proof}
	We will prove this result by induction on $p$.
	
	Since $\gamma\neq\delta$ our base case is $p=1$ thus
	we have two possibilities: either $p_0=1$ and $p_1=0$ or $p_0=0$ and $p_1=1$.
	If $p_0=1,p_2=0$ then $\delta$ is the 0-reduction of $\gamma$ i.e.
	\[
	\delta=C_{0}C_{1}\dots C_{m-1}
	\]
	and if $p_0=0, p_2=1$ then $\delta$ is the 1-reduction of $\gamma$:
	\[
	\delta=C_{1}C_{2}\dots C_{m}
	\]
	in both cases we have the stated result.
	
	Now let $p>1$. Then we have a 0-1-sequence
	$\gamma_0,\gamma_1,\dots, \gamma_{p-1},\gamma_p$ such that $\gamma_0=\gamma$ and $\gamma_p=\delta$.
	Suppose $p-1=p'_0+p'_1$ where $p'_0,p'_1$ are respectively the number of 0-reductions and 1-reductions in the sequence
	$\gamma_0,\gamma_1,\dots, \gamma_{p-1}$.
	By inductive hypothesis:
	\[
	\gamma_{p-1}=C_{p'_1}C_{p_2+1}\dots C_{m-p'_0}
	\]
	Since $p=p_0+p_1$ we must have that either $p_0=p_0'+1$ or $p_1=p_1'+1$.
	If $p_0=p_0'+1$ then $\gamma_p$ is the 0-reduction of $\gamma_{p-1}$ and $p_1=p_1'$:
	\begin{align*}
	\gamma_{p}=\delta	&=C_{p'_1}C_{p'_1+1}\dots C_{m-p'_0-1}\\
						&=C_{p_1}C_{p_1+1}\dots C_{m-(p'_0+1)}\\
						&=C_{p_1}C_{p_1+1}\dots C_{m-p_0}
	\end{align*}
	If, on the other hand, $p_1=p_1'+1$ then $\gamma_p$ is the 1-reduction of $\gamma_{p-1}$ and $p_0=p_0'$:
	\begin{align*}
	\gamma_{p}=\delta	&=C_{p_1'+1}C_{p'_1+2}\dots C_{m-p'_0}\\
	&=C_{p_1}C_{p_1+1}\dots C_{m-p_0}
	\end{align*}
	Hence in either case we conclude the desired result.

\end{proof}

Albeit 0-1-sequences for a given $\delta\in X_\gamma$ do not need to be unique,
they all have the same lenght.	
\begin{lema} \label{le:LongitudUnica}
	Let $\gamma$ be a non trivial PNC and $\delta\in X_\gamma$
	with $\delta\neq \gamma$.
	
	If $\gamma_0,\dots \gamma_p$ and $\delta_0,\ldots,\delta_q$
	are 0-1-sequences for $\delta$
	then $p=q$.
\end{lema}

\begin{proof}
	We will prove that if there are 0-1-sequences $\gamma_0,\gamma_1,\dots, \gamma_p$ and $\delta_0,\delta_1,\ldots, \delta_q$
	in $X_\gamma$ with $p\neq q$ then $\gamma_p\neq \delta_q$.

	Let $p_0$ and $p_1$ be respectively the number of 0-reductions
	and 1-reductions  in the sequence $\gamma_0,\dots, \gamma_p$. On the other hand let
	$q_0,q_1$ be the number of 0-reductions
	and 1-reductions respectively in the sequence $\delta_0,\dots, \delta_q$.
	If $p<q$ then we have either $p_1<q_1$ or $p_1<q_2$. In both cases
	the conclusion of Lemma \ref{teo:TeoremaIndependencia} guarantees that $\gamma_p\neq\delta_q$.
	The case for $q<p$ is analogous.

\end{proof}

\begin{definicion}
	Let $\gamma$ be a PNC, $\delta\in X_\gamma$ and $\gamma_0,\gamma_1,\dots, \gamma_p$
	a 0-1-sequence for $\delta$.
	
	The {\em characteristic sequence} of $\gamma_0,\gamma_1,\dots, \gamma_p$
	is the binary sequence $s_1,\dots, s_p$ such that $s_i=0$ if $\gamma_i$
	is the 0-reduction of $\gamma_{i-1}$ and $s_i=1$ if $\gamma_i$
	is the 1-reduction of $\gamma_{i-1}$.
\end{definicion}

\begin{teorema}\label{teo:Isomorfismo}
	If $\gamma$ is a perfectly nested circuit with $m$ internal vertices,
	then the partial orders $(X_\gamma,\leq_\gamma)$ and $(S_m,\leq_m)$
	are isomorphic.
\end{teorema}

\begin{proof}
	Define a function $f:X_\gamma\longrightarrow S_m$ as follows:
	\begin{equation}
	f(\delta)=
	\begin{cases}
	[\emptyset]	&	\text{if}\quad \delta=\gamma\\
	[s]\quad \text{$s$ the characteristic sequence of a 0-1 sequence for $\delta$}			
	&	\text{if}\quad \delta\neq\gamma		
	\end{cases}
	\end{equation}
	By Lemma \ref{le:ObtencionDeCircuitos} there is a 0-1 sequence for each $\delta\in X_\gamma$.
	By Lemmas \ref{teo:TeoremaIndependencia} and \ref{le:LongitudUnica} this function is well defined
	i.e. it does not depend on the 0-1 sequence we choose for $\delta$.

	If $\delta$ and $\delta'$ are two circuits with $f(\delta)=f(\delta')$
	Lemma \ref{teo:TeoremaIndependencia} says that $\delta=\delta'$ thus $f$ is injective.

	On the other hand consider $[s]\in S_m$ with $\ell(s)=p\leq m$ such that $|s|_0=p_0$ and $|s|_1=p_1$.
	If $\delta\in X_\gamma$ is obtained from $\gamma$
	by $p_0$ $0$-reductions and $p_1$ $1$-reductions,
	then $f(\delta)=[s]$. Hence
	$f$ is onto.

	Therefore $f$ is a bijection between $X_\gamma$ and $S_m$.

	Now we are going to prove that
		\[
		\xi\leq_\gamma \delta \iff f(\xi)\leq_m f(\delta)
		\]
		for every pair of circuits $\xi,\delta\in X_\gamma$.

	Given $\xi$ and $\delta$ in $X_\gamma$,
	$\xi\leq_\gamma \delta$ if and only if $\delta\stackrel{*}{\rightarrow}_C\xi$
	thus $\xi\leq_\gamma \delta$ if and only if $\xi\in X_\delta$.
	Hence $\xi\leq_\gamma \delta$ is equivalent to the existence of a
	0-1-sequence $\gamma_0,\dots, \gamma_p$ for $\delta$ en $X_\gamma$
	and a 0-1-sequence $\delta_0,\dots, \delta_q$ for $\xi$ in $X_\delta$.

	Since $X_\delta\subseteq X_\gamma$ (Lemma \ref{le:XGammaEsTransitivo})
	we have that every term in
	the sequence $\{\delta_i\}$ is also an element of $X_\gamma$.
	Thus saying that $\xi\leq_\gamma \delta$ is equivalent to say that
	there is a 0-1-sequence $\gamma_0,\dots, \gamma_p=\delta_0,\dots, \delta_q$
	for $\xi$ in $X_\gamma$.
	
	It is immediate that if $t$ is the characteristic sequence of the 0-1 sequence ${\gamma_i}$
 	and $s$ is the characteristic sequence of
 $\gamma_0,\dots, \gamma_p=\delta_0,\dots, \delta_q$ then $s$ extends $t$.

 Since $f(\delta)=[s]$ and $f(\xi)=[t]$ we have
 $\xi\leq_\gamma \delta$ if and only if $f(\delta)\leq_mf(\xi)$,
 as we wanted to prove.		
	
\end{proof}

\section{Conclusions and further research}\label{sec:Conclusiones}
In this article, a kind of nested graph is presented along with its properties. It was called perfectly nested circuit after analyzing the characteristics as graph structure. Several concepts, definitions and figures were used to make clear the identification of this type of circuits. As a result we establish an order isomorphism between some sets of perfectly nested circuits and equivalence classes over finite binary sequences. Considering similarities with nested graphs the next step should include the use of perfectly nested circuits in knowledge representation, semantic of natural languages and inference in databases.


\end{document}